\documentclass{article}

\usepackage{times}
\usepackage{graphicx} % more modern
\usepackage{subfigure} 

\usepackage{natbib}

\usepackage{algorithm}
\usepackage{algorithmic}

\usepackage{hyperref}

\usepackage[accepted]{icml2015}

\icmltitlerunning{Faster Rates for the Frank-Wolfe Method over Strongly-Convex Sets}

\usepackage{amsfonts}
\usepackage{amsmath}
\usepackage{amsthm}
\usepackage{comment}

\newtheorem{theorem} {Theorem}
\newtheorem{lemma} {Lemma}
\newtheorem{definition} {Definition}
\newtheorem{corollary} {Corollary}

\newcommand{\mK}{\mathcal{K}}

\newcommand{\ball}{\mathbb{B}}

\newcommand{\vecspace}{\textbf{E}}
\newcommand{\Rn}{\mathbb{R}^n}
\newcommand{\Rmn}{\mathbb{R}^{m\times{n}}}

\newcommand{\oraclek}{\mathcal{O}_{\mK}}
\newcommand{\nnz}{\textrm{nnz}}

\begin{document} 

\twocolumn[
\icmltitle{Faster Rates for the Frank-Wolfe Method over Strongly-Convex Sets}

% It is OKAY to include author information, even for blind
% submissions: the style file will automatically remove it for you
% unless you've provided the [accepted] option to the icml2015
% package.
\icmlauthor{Dan Garber}{dangar@tx.technion.ac.il}
\icmladdress{Technion - Israel Institute of Technology}
\icmlauthor{Elad Hazan}{ehazan@cs.princeton.edu}
\icmladdress{Princeton University}

% You may provide any keywords that you 
% find helpful for describing your paper; these are used to populate 
% the "keywords" metadata in the PDF but will not be shown in the document
\icmlkeywords{boring formatting information, machine learning, ICML}

\vskip 0.3in
]

\begin{abstract} 
The Frank-Wolfe method (a.k.a. conditional gradient algorithm) for smooth optimization has regained much interest in recent years in the context of large scale optimization and machine learning. A key advantage of the method is that it avoids projections - the computational bottleneck in many applications - replacing it by a linear optimization step. Despite this advantage, the known convergence rates of the FW method fall behind standard first order methods for most settings of interest. It is an active line of research to derive faster linear optimization-based algorithms for various settings of convex optimization. 

In this paper we consider the special case of optimization over strongly convex sets, for which we prove that the vanila FW method converges at a rate of  $\frac{1}{t^2}$. This gives a quadratic improvement in convergence rate compared to the general case, in which convergence is of the order $\frac{1}{t}$, and known to be tight.
We show that various balls induced by $\ell_p$ norms, Schatten norms and group norms are strongly convex on one hand and on the other hand, linear optimization over these sets is straightforward and admits a closed-form solution.
We further show how several previous fast-rate results for the FW method follow easily from our analysis.
\end{abstract} 

\section{Introduction}

The Frank-Wolfe method, originally introduced by Frank and Wolfe in the 1950's \cite{FrankWolfe}, is a \textit{first order} method for the minimization of a smooth convex function over a convex set. Its main advantage in large-scale problems is that it is a first-order and projection-free method - i.e. the algorithm proceeds by iteratively solving a linear optimization problem and remaining inside the feasible domain. For matrix completion problems, metric learning, sparse PCA, structural SVM and other large-scale machine learning problems, this feature enabled the derivation of algorithms that are practical on one hand and come with provable convergence rates on the other  \cite{Jaggi10, Jaggi13a, Dudik12a, Dudik12b, Hazan12, ShalevShwartz11, Laue12}.

Despite its empirical success, the main drawback of the method is its relatively slow convergence rate in comparison to optimal first order methods. The convergence rate of the method is on the order of $1/t$ where $t$ is the number of iterations, and this is known to be tight. In contrast, Nesterov's accelerated gradient descent method gives a  rate of $1/t^2$  for general convex smooth problems and a rate $e^{-\Theta(t)}$ is known for smooth and strongly convex problems.  The following question arises: are there projection-free methods with convergence rates matching that of projected gradient-descent and its extensions?

Motivated by this question, in this work we advance the line of research for faster convergence rates of projection free methods. We prove that in case both the objective function and the feasible set are strongly convex (in fact a slightly weaker assumption than strong convexity of the objective is required), the vanilla Frank-Wolfe method converges at an accelerated rate of  $1/t^2$. The improved convergence rate is independent of the dimension. This is also the first convergence result for the FW that we are aware of that achieves a rate that is between the standard $1/t$ rate and a linear rate.
We further show how the analysis used to prove the latter result enables to easily derive previous fast convergence rates for the FW method.

We motivate the study of optimization over strongly convex sets by demonstrating that various norms that serve as popular regularizes in machine learning problems, including $\ell_p$ norms, matrix Schatten norms and matrix group norms, give rise to strongly convex sets. We further show that indeed linear optimization over these sets is straightforward to implement and admits a closed-form solution. Hence the FW method is appealing for solving optimization problems with such constraints, such as regularized linear regression.

\subsection{Related Work}

The Frank-Wolfe method dates back to the original work of Frank and Wolfe \cite{FrankWolfe} which presented an algorithm for minimizing a quadratic function over a polytope using only linear optimization steps over the feasible set. Recent results by Clarkson \cite{Clarkson}, Hazan \cite{Hazan08} and Jaggi \cite{Jaggi13b} extend the method to smooth convex optimization over the \textit{simplex}, \textit{spectrahedron} and arbitrary convex and compact sets respectively. 

It was shown in numerous works that the convergence rate of the method is on the order of $1/t$ and that it could not be improved in general, even if the objective function is strongly convex for instance, as shown in \cite{Clarkson, Hazan08, Jaggi13b}, even though it is known that in this case, the projected gradient method achieves an exponentially fast convergence rate.  

Over the past years, several results tried to improve the convergence rate of the Frank-Wolfe method under various assumptions. Gu{\'{e}}Lat and Marcotte \cite{GueLat1986} showed that in case the objective function is strongly convex and the feasible set is a polytope, then in case the optimal solution is located in the interior of the set, the FW method converges exponentially fast. A similar result was presented in the work of Beck and Teboulle \cite{BeckTaboule} who considered a specific problem they refer to a \textit{the convex feasibility problem} over an arbitrary convex set. They also obtained a linear convergence rate under the assumption that an optimal solution that is far enough from the boundary of the set exists.

Recently, Garber and Hazan \cite{Garber13} gave the first natural linearly-converging FW variant without any restricting assumptions on the location of the optimum. They showed that a variant of the Frank Wolfe method with the addition of \textit{away steps} converges exponentially fast in case the objective function is strongly convex and the feasible set is a polytope. 
In follow-up work, Jaggi and Lacoste-Julien \cite{Jaggi13c} gave a refined analysis of an algorithm presented in \cite{GueLat1986} which also uses away steps and showed that it also converges exponentially fast in the same setting as the Garber-Hazan result. Also relevant in this context is the work of Ahipasaoglu, Sun and Todd \cite{Ahipasaoglu08} who showed that in the specific case of minimizing a smooth and strongly convex function over the unit simplex, a variant of the Frank-Wolfe method that also uses away steps converges with a linear rate.

In a different line of work, Migdalas and recently Lan \cite{Migdalas, Lan13} considered the Frank-Wolfe algorithm with a stronger optimization oracle that is able to solve quadratic problems over the feasible domain. They show that in case the objective function is strongly convex then exponentially fast convergence is attainable. However, in most settings of interest, an implementation of such a non-linear oracle is computationally much more expensive than the linear oracle, and the key benefit of the Frank-Wolfe method is lost. 

In the specific case that the feasible set is strongly convex, an assumption also made in this paper, Levitin and Polyak showed in their classical work \cite{Polyak} that under the restrictive assumption that the norm of the gradient of the objective function is lower bounded by a constant everywhere in the feasible set, the FW method converges with an exponential rate. The same result appeared in following works by Demyanov and Rubinov \cite{Demyanov70} and Dunn \cite{Dunn79}, both also requiring that the magnitude of the gradients is lower bounded by a constant everywhere in the feasible set. As we later show, the lower bound requirement on the gradients is in a sense much stronger than requiring that the objective function is strongly convex.  Under our assumption however, which is slightly weaker than strong convexity of the objective, the gradient may become arbitrarily small on the feasible set.

We summarize previous convergence rate results for the standard FW method in Table \ref{table:prevwork}.

\begin{table*}\label{table:prevwork}
\begin{center}
  \begin{tabular}{| c | c | c | c | c | c |}
    \hline
    Reference & Feasible set $\mK$ & Objective function $f$ & Location of $x^*$ & Conv. rate \\ \hline
    \cite{Jaggi13b} & convex & convex & unrestricted  & $1/t$ \\ \hline
    \cite{GueLat1986} & polytope & strongly convex & interior  & $\exp(-\Theta(t))$ \\ \hline
    \cite{BeckTaboule} & convex & $f(x) = \Vert{Ax-b}\Vert_2^2$ & interior & $\exp(-\Theta(t))$ \\ \hline
    \cite{Polyak} &  &   &  &  \\ 
    \cite{Demyanov70} & strongly convex & $\Vert{\nabla{}f(x)}\Vert \geq c > 0 \quad \forall{x\in\mK}$  & unrestricted & $\exp(-\Theta(t))$ \\ 
    \cite{Dunn79} & &   & &  \\ \hline

    this paper & strongly convex & strongly convex  & unrestricted & $1/t^2$ \\ \hline
    
  \end{tabular}
  \caption{Comparison of convergence rates for the Frank-Wolfe method under different assumptions. We denote the optimal solution by $x^*$. Note that since all results assume smoothness of the function we omit it from column 3.}
\end{center}
\end{table*}

\section{Preliminaries}\label{sec:preliminaries}

\subsection{Smoothness and Strong Convexity}

For the following definitions let $\vecspace$ be a finite vector space and $\Vert\cdot\Vert$, $\Vert\cdot\Vert_*$  be a pair of dual norms over $\vecspace$.

\begin{definition}[smooth function]\label{def:smoothfunc}
We say that a function $f:\vecspace\rightarrow\mathbb{R}$ is $\beta$ smooth over a convex set $\mK\subset\vecspace$  with respect to $\Vert\cdot\Vert$ if for all $x,y\in\mK$ it holds that
\begin{eqnarray*}
f(y) \leq f(x) + \nabla{}f(x)\cdot(y-x) + \frac{\beta}{2}\Vert{x-y}\Vert^2 .
\end{eqnarray*}
\end{definition}

\begin{definition}[strongly convex function]\label{def:strongconvexfunc}
We say that a function $f:\vecspace\rightarrow\mathbb{R}$ is $\alpha$-strongly convex over a convex set $\mK\subset\vecspace$ with respect to $\Vert\cdot\Vert$ if it satisfies the following two equivalent conditions
 
\begin{enumerate}
\item $\forall{x,y\in\mK}:$
\begin{eqnarray*}
f(y) \geq f(x) + \nabla{}f(x)\cdot(y-x) + \frac{\alpha}{2}\Vert{x-y}\Vert^2 .
\end{eqnarray*}
\item $\forall{x,y\in\mK, \gamma\in[0,1]}:$
\begin{eqnarray*}
f(\gamma{}x+(1-\gamma)y) &\leq &\gamma{}f(x) + (1-\gamma)f(y) \\
&-&\frac{\alpha}{2}\gamma(1-\gamma)\Vert{x-y}\Vert^2 .
\end{eqnarray*}

\end{enumerate}

\end{definition}

The above definition (part 1) combined with first order optimality conditions imply that for a $\alpha$-strongly convex function $f$, if $x^*=\arg\min_{x\in\mK}f(x)$, then for any $x\in\mK$
\begin{eqnarray}\label{ie:strongconvex}
f(x)-f(x^*) \geq \frac{\alpha}{2}\Vert{x-x^*}\Vert^2 .
\end{eqnarray}

Eq. \eqref{ie:strongconvex} further implies that the magnitude of the gradient of $f$ at point $x$,  $\Vert{\nabla{}f(x)}\Vert_*$ is  at least of the order of the square-root of the approximation error at $x$, $f(x)-f(x^*)$. This follows since
\begin{align*}
\sqrt{\frac{2}{\alpha}\left({f(x)-f(x^*)}\right)}\cdot\Vert{\nabla{}f(x)}\Vert_* &\geq \Vert{x-x^*}\Vert\cdot\Vert{\nabla{}f(x)}\Vert_* \\
&\geq (x-x^*)\cdot\nabla{}f(x)\\
& \geq  f(x) - f(x^*) ,
\end{align*}
where the first inequality follows from \eqref{ie:strongconvex}, the second from Holder's inequality and the third from convexity of $f$. Thus we have that at any point $x\in\mK$ it holds that
\begin{eqnarray}\label{ie:largegrad}
\Vert{\nabla{}f(x)}\Vert_* \geq \sqrt{\frac{\alpha}{2}}\cdot\sqrt{f(x)-f(x^*)} .
\end{eqnarray}
We will show that this property, that is in fact weaker than strong convexity, combined with an additional property of the convex set that we define next, allows to obtain the faster rates \footnote{In this work we assume that the convex set $\mK$ is full-dimensional. In case this assumption does not hold, e.g. if the convex set is the unit simplex, then Eq. \eqref{ie:largegrad} holds even if we replace $\nabla{}f(x)$ with $P_{S(\mK)}[\nabla{}f(x)]$ where $P_{S(\mK)}$ denotes the projection operator onto the smallest subspace that contains $\mK$.}.
\begin{definition}[strongly convex set]\label{def:strongconvexset}
We say that a convex set $\mK\subset\vecspace$ is $\alpha$-strongly convex with respect to $\Vert\cdot\Vert$ if for any $x,y\in\mK$, any $\gamma\in[0,1]$ and any vector $z\in\vecspace$ such that $\Vert{z}\Vert=1$, it holds that
\begin{eqnarray*}
\gamma{}x + (1-\gamma)y + \gamma(1-\gamma)\frac{\alpha}{2}\Vert{x-y}\Vert^2z\in\mK .
\end{eqnarray*}
That is, $\mK$ contains a ball of  of radius $\gamma(1-\gamma)\frac{\alpha}{2}\Vert{x-y}\Vert^2$ induced by the norm $\Vert\cdot\Vert$ centered at $\gamma{}x + (1-\gamma)y$.
\end{definition}

\subsection{The Frank-Wolfe Algorithm}

The Frank-Wolfe algorithm, also known as the \textit{conditional gradient algorithm}, is an algorithm for the minimization of a convex function $f:\vecspace\rightarrow\mathbb{R}$ which is assumed to be $\beta_f$-smooth with respect to a norm $\Vert\cdot\Vert$, over a convex and compact set $\mK\subset\vecspace$. The algorithm implicitly assumes that the convex set $\mK$ is given in terms of a linear optimization oracle $\oraclek:\vecspace\rightarrow\mK$ which given a linear objective $c\in\vecspace$ returns a point $x=\oraclek(c)\in\mK$ such that $x\in\arg\min_{y\in\mK}y\cdot{}c$. The algorithm is given below. The algorithm proceeds in iterations, taking on each iteration $t$ the new iterate $x_{t+1}$ to be a convex combination between the previous feasible iterate $x_t$ and a feasible point that minimizes the dot product with the gradient direction at $x_t$, which is generated by invoking the oracle $\oraclek$ with the input vector $\nabla{f}(x_t)$. There are various ways to set the parameter that controls the convex combination $\eta_t$ in order to guarantee convergence of the method. The option that we choose here is the optimization of $\eta_t$ via a simple line search rule, which is straightforward and computationally cheap to implement.

\begin{algorithm}[h]
\caption{Frank-Wolfe Algorithm}
\label{alg:condgrad}
\begin{algorithmic}[1]
\STATE Let $x_0$ be an arbitrary point in $\mK$.
\FOR{$t = 0,1,...$}
\STATE $p_{t} \gets \oraclek(\nabla{}f(x_t))$.
\STATE  $\eta_t \gets \arg\min_{\eta\in[0,1]}\eta(p_t - x_t)\cdot\nabla{}f(x_t) + \eta^2\frac{\beta_f}{2}\Vert{p_t-x_t}\Vert^2$.
\STATE $x_{t+1} \gets x_t + \eta_t(p_t-x_t)$.
\ENDFOR
\end{algorithmic}
\end{algorithm}

The following theorem states the well-known convergence rate of the Frank-Wolfe algorithm for smooth convex minimization over a compact and convex set, without any further assumptions. A proof is given in the appendix for completeness though similar proofs could also be found in \cite{Polyak, Jaggi13b}.
\begin{theorem}\label{thr:originalfw}
Let $x^*\in\arg\min_{x\in\mK}f(x)$ and denote $D_{\mK}=\max_{x,y\in\mK}\Vert{x-y}\Vert$ (the diameter of the set with respect to $\Vert\cdot\Vert$). For every $t\geq 1$ the iterate $x_t$ of Algorithm \ref{alg:condgrad} satisfies
\begin{eqnarray*}
f(x_t) - f(x^*) \leq \frac{8\beta_fD_{\mK}^2}{t} = O\left({\frac{1}{t}}\right) .
\end{eqnarray*}
\end{theorem}

\subsection{Our Results}
In this work, we consider the case in which the function to optimize $f$ is not only $\beta_f$-smooth with respect to $\Vert\cdot\Vert$ but also $\alpha_f$-strongly convex with respect to $\Vert\cdot\Vert$ (we relax this assumption a bit in subsection \ref{sec:relax}). We further assume that the feasible set $\mK$ is $\alpha_{\mK}$-strongly convex with respect to $\Vert\cdot\Vert$. Under these two additional assumptions alone we prove the following theorem.

\begin{theorem}\label{thr:newfw}
Let $x^*=\arg\min_{x\in\mK}f(x)$ and let $M =\frac{\sqrt{\alpha_f}\alpha_K}{8\sqrt{2}\beta_f}
$ . Denote $D_{\mK}=\max_{x,y\in\mK}\Vert{x-y}\Vert$. For every $t \geq 1$ the iterate $x_t$ of Algorithm \ref{alg:condgrad} satisfies
\begin{eqnarray*}
f(x_t)-f(x^*) \leq \frac{\max\lbrace{\frac{9}{2}\beta_fD_{\mK}^2, 18M^{-2}}\rbrace}{(t+2)^2} = O\left({\frac{1}{t^2}}\right).
\end{eqnarray*}
\end{theorem}

\section{Proof of Theorem \ref{thr:newfw}}

We denote the approximation error of the iterate $x_t$ produced by the algorithm by $h_t$. That is $h_t = f(x_t) - f(x^*)$ where $x^* = \arg\min_{x\in\mK}f(x)$.

To better illustrate our results, we first shortly revisit the proof technique of Theorem \ref{thr:originalfw}. The main observation to be made is the following:
\begin{eqnarray}\label{old_fw_anal}
&&  h_{t+1} = f(x_t + \eta_t(p_t - x_t)) - f(x^*)  \leq \nonumber\\
&&  h_t + \eta_t(p_t-x_t)\cdot\nabla{}f(x_t) + \frac{\eta_t^2\beta_f}{2}\Vert{p_t-x_t}\Vert^2 \leq \nonumber \\ 
&& h_t + \eta_t(x^*-x_t)\cdot\nabla{}f(x_t) + \frac{\eta_t^2\beta_f}{2}\Vert{p_t-x_t}\Vert^2 \leq \nonumber \\ 
&& (1 - \eta_t)h_t + \frac{\eta_t^2\beta_f}{2}\Vert{p_t-x_t}\Vert^2 ,
\end{eqnarray}
where the the first inequality follows from the smoothness of $f$, the second from the optimality of $p_t$ and the third from convexity of $f$.
Choosing $\eta_t$ to be roughly $1/t$ yields the convergence rate of $1/t$ stated in Theorem \ref{thr:originalfw}. This rate cannot be improved in general since while the so-called \textit{duality gap} $(x_t-p_t)\cdot\nabla{}f(x_t)$ could be arbitrarily small (as small as $(x_t-x^*)\cdot\nabla{}f(x_t)$), the quantity $\Vert{p_t-x_t}\Vert$ may remain as large as the diameter of the set. Note that in case $f$ is strongly-convex, then according to Eq. \eqref{ie:strongconvex} it holds that $x_t$ converges to $x^*$ and thus according to Eq. \eqref{old_fw_anal} it suffices to solve the inner linear optimization problem in Algorithm \ref{alg:condgrad} on the intersection of $\mK$ and a small ball centered at $x_t$. As a result the quantity $\Vert{p_t-x_t}\Vert^2$ will be proportional to the approximation error at time $t$, and a linear convergence rate will be attained. However in general, under the linear oracle assumption, we have no way to solve the linear optimization problem over the intersection of $\mK$ and a ball without greatly increasing the number of calls to the linear oracle, which is the most expensive step in many settings.

In case the feasible set $\mK$ is strongly convex, then the main observation to be made is that  while the quantity $\Vert{p_t-x_t}\Vert$ may still be much larger than $\Vert{x^*-x_t}\Vert$ (the distance to the optimum), in this case, the \textit{duality gap} must also be large, which results in faster convergence. This observation is illustrated in Figure \ref{fig:fw} and given formally in Lemma \ref{lem:fwineq}.
 
\begin{figure}
\centering
\includegraphics[width=0.51\textwidth]{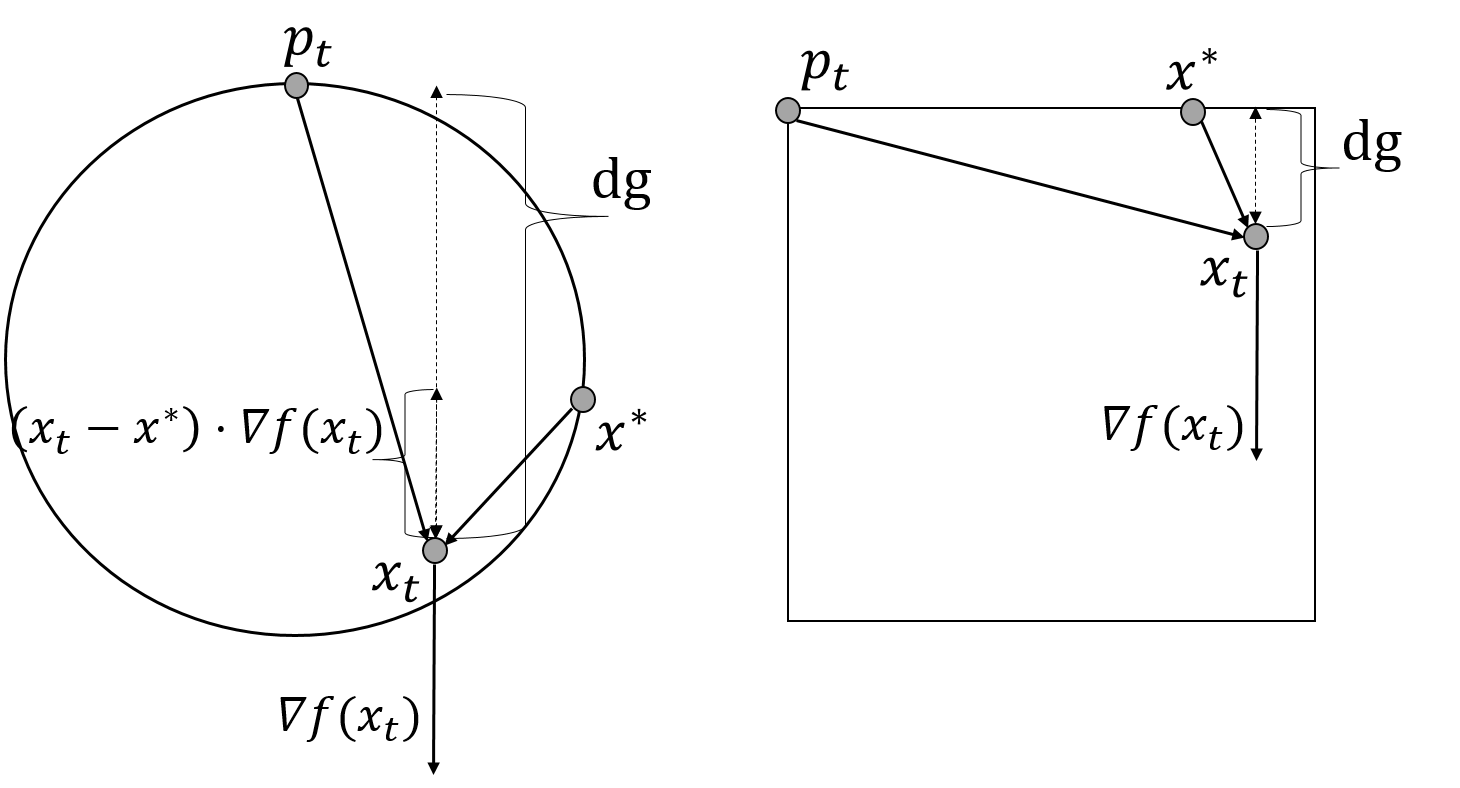}
\caption{For strongly convex sets, as in the left picture, the duality gap (denoted dg) increases with $\Vert{p_t-x_t}\Vert^2$, which accelerates the convergence of the Frank-Wolfe method. As shown in the picture on the right, this property clearly does not hold for arbitrary convex sets.}
\label{fig:fw}
\end{figure}

\begin{lemma}\label{lem:fwineq}
On any iteration $t$ of Algorithm \ref{alg:condgrad} it holds that
\begin{eqnarray*}
h_{t+1} \leq h_t\cdot\max\lbrace{\frac{1}{2}, 1-\frac{\alpha_K\Vert{\nabla{}f(x_t)}\Vert_*}{8\beta_f}}\rbrace .
\end{eqnarray*}
\end{lemma}
\begin{proof}
By the optimality  of the point $p_t$ we have that
\begin{eqnarray}\label{inq:steplemma_1}
(p_t - x_t)\cdot\nabla{}f(x_t) &\leq & (x^* - x_t)\cdot\nabla{}f(x_t) \nonumber \\
&\leq & f(x^*) - f(x_t) = -h_t ,
\end{eqnarray}
where the second inequality follows from convexity of $f$.
Denote $c_t = \frac{1}{2}(x_t + p_t)$ and $w_t \in \arg\min_{w\in\vecspace,\Vert{w}\Vert\leq 1}w\cdot\nabla{}f(x_t)$.  Note that from Holder's inequality we have that $w_t\cdot\nabla{}f(x_t) = -\Vert{\nabla{}f(x_t)}\Vert_*$.
Using the strong convexity of the set $\mK$ we have that the point
$\tilde{p}_t = c_t + \frac{\alpha_K}{8}\Vert{x_t-p_t}\Vert^2w_t$ is in $\mK$.
Again using the optimality of $p_t$ we have that
\begin{align}\label{ie:errordec}
&(p_t - x_t)\cdot\nabla{}f(x_t) \leq  (\tilde{p}_t - x_t)\cdot\nabla{}f(x_t) \nonumber \\
& = \frac{1}{2}(p_t-x_t)\cdot\nabla{}f(x_t)  + \frac{\alpha_K\Vert{x_t-p_t}\Vert^2}{8}w_t\cdot\nabla{}f(x_t) \nonumber \\
& \leq -\frac{1}{2}h_t - \frac{\alpha_K\Vert{x_t-p_t}\Vert^2}{8}\Vert{\nabla{}f(x_t)}\Vert_* ,
\end{align}
where the last inequality follows from Eq. \eqref{inq:steplemma_1}.

We now analyze the decrease in the approximation error $h_{t+1}$. By smoothness of $f$ we have
\begin{eqnarray*}
f(x_{t+1}) &\leq &f(x_t) + \eta_t(p_t - x_t)\cdot\nabla{}f(x_t) \\
&+& \frac{\beta_f}{2}\eta_t^2\Vert{p_t - x_t}\Vert^2 .
\end{eqnarray*}
Subtracting $f(x^*)$ from both sides we have
\begin{align}\label{ie:errordec2}
h_{t+1} &\leq &h_t + \eta_t(p_t - x_t)\cdot\nabla{}f(x_t) + \frac{\beta_f}{2}\eta_t^2\Vert{p_t - x_t}\Vert^2 .\nonumber \\
\end{align}
Plugging Eq. \eqref{ie:errordec} we have
\begin{eqnarray*}
h_{t+1} &\leq & h_t\left({1-\frac{\eta_t}{2}}\right) - \eta_t\frac{\alpha_K\Vert{x_t-p_t}\Vert^2}{8}\Vert{\nabla{}f(x_t)}\Vert_*  \nonumber \\
& + & \frac{\beta_f}{2}\eta_t^2\Vert{p_t - x_t}\Vert^2 \\
&=& h_t\left({1-\frac{\eta_t}{2}}\right) \\
&+& \frac{\Vert{x_t-p_t}\Vert^2}{2}\left({\eta_t^2\beta_f - \eta_t\frac{\alpha_K\Vert{\nabla{}f(x_t)}\Vert_*}{4}}\right) .
\end{eqnarray*}
In case $\frac{\alpha_K\Vert{\nabla{}f(x_t)}\Vert_*}{4} \geq \beta_f$, by the optimal choice of $\eta_t$ in Algorithm \ref{alg:condgrad}, we can set $\eta_t =1$ and get
\begin{eqnarray*}
h_{t+1} \leq \frac{h_t}{2} .
\end{eqnarray*}

Otherwise, we can set $\eta_t = \frac{\alpha_K\Vert{\nabla{}f(x_t)}\Vert_*}{4\beta_f}$ and get
\begin{eqnarray*}
h_{t+1} &\leq & h_t\left({1-\frac{\alpha_K\Vert{\nabla{}f(x_t)}\Vert_*}{8\beta_f}}\right) .
\end{eqnarray*}
\end{proof}

Note that Lemma \ref{lem:fwineq} only relies on the strong convexity of the set $\mK$ and did not assume anything regrading $f$ beyond convexity and smoothness. In particular it does not require $f$ to be strongly convex.

We can now prove Theorem \ref{thr:newfw}.
\begin{proof}
Let $M =\frac{\sqrt{\alpha_f}\alpha_K}{8\sqrt{2}\beta_f}$ and $C = \max\lbrace{\frac{9}{2}\beta_fD_{\mK}^2, 18M^{-2}}\rbrace$. We prove by induction that for all $t\geq 1$, $h_t \leq \frac{C}{(t+2)^2}$. 

Since we assume that the objective function $f$ satisfies Eq. \eqref{ie:largegrad}, we have from Lemma \ref{lem:fwineq} that on any iteration $t$,
\begin{eqnarray}\label{ie:mainproof}
h_{t+1} &\leq & h_t\cdot\max\lbrace{\frac{1}{2}, 1-\frac{\alpha_K\sqrt{\alpha_f}}{8\sqrt{2}\beta_f}\sqrt{h_t}}\rbrace  \nonumber \\
&=& h_t\cdot\max\lbrace{\frac{1}{2}, 1-Mh_t^{1/2}}\rbrace .
\end{eqnarray}
For the base case $t=1$ we need to prove that $h_1=f(x_1)-f(x^*) \leq C/4$. 
By $\beta_f$ smoothness of $f$ we have
\begin{align*}
f(x_1) - f(x^*) &= f(x_0 + \eta_0(p_0 - x_0)) - f(x^*)\\
&\leq  h_0  + \eta_0(p_0-x_0)\cdot\nabla{}f(x_0) + \frac{\beta_f\eta_0^2}{2}D_{\mK}^2 \\
&\leq h_0(1-\eta_0) + \frac{\beta_f\eta_0^2}{2}D_{\mK}^2 ,
\end{align*}
where the last inequality follows from convexity of $f$. By the optimal choice of $\eta_0$ we can in particular set $\eta_0=1$ which gives $h_1 \leq \frac{\beta_f}{2}D_{\mK}^2 \leq C/9$.

Assume now that the induction holds for time $t\geq 1$, that is $h_t \leq \frac{C}{(t+2)^2}$.

If the result of the $\max$ operation in Eq. \eqref{ie:mainproof} is the first argument, that is $1/2$, we have that 
\begin{eqnarray}\label{ie:mainproof2}
h_{t+1} &\leq & \frac{h_t}{2} \leq \frac{C}{2(t+2)^2} = \frac{C}{(t+3)^2}\cdot\frac{(t+3)^2}{2(t+2)^2} \nonumber \\
& \leq & \frac{C}{(t+3)^2} .
\end{eqnarray}
where the last inequality holds for any $t \geq 1$.

We now turn to the case in which the result of the $\max$ operation in Eq. \eqref{ie:mainproof} is the second argument. We consider two cases.

If $h_t \leq \frac{C}{2(t+2)^2}$ then as in Eq. \eqref{ie:mainproof2} it holds for any $t\geq 1$ that
\begin{eqnarray*}
h_{t+1} \leq h_t \leq \frac{C}{2(t+2)^2} \leq \frac{C}{(t+3)^2} ,
\end{eqnarray*}
where the first inequality follows from Eq. \eqref{ie:mainproof}.

Otherwise,  $h_t > \frac{C}{2(t+2)^2}$. By Eq. \eqref{ie:mainproof} and the induction assumption we have
\begin{align*}
h_{t+1} &\leq  h_t\left({1-Mh_t^{1/2}}\right) \\
&<  \frac{C}{(t+2)^2}\left({1-M\sqrt{\frac{C}{2}}\frac{1}{t+2}}\right)  \\
&= \frac{C}{(t+3)^2}\cdot\frac{(t+3)^2}{(t+2)^2}\left({1-M\sqrt{\frac{C}{2}}\frac{1}{t+2}}\right) \\
&= \frac{C}{(t+3)^2}\cdot\frac{(t+2)^2+2t+5}{(t+2)^2}\left({1-M\sqrt{\frac{C}{2}}\frac{1}{t+2}}\right)  \\
&<  \frac{C}{(t+3)^2}\left({1+\frac{3(t+2)}{(t+2)^2}}\right)\left({1-M\sqrt{\frac{C}{2}}\frac{1}{t+2}}\right)  \\
&=  \frac{C}{(t+3)^2}\left({1+\frac{3}{t+2}}\right)\left({1-M\sqrt{\frac{C}{2}}\frac{1}{t+2}}\right) .
\end{align*}
Thus for $C \geq \frac{18}{M^2}$ we have that
\begin{align*}
h_{t+1} &\leq  \frac{C}{(t+3)^2}\left({1+\frac{3}{t+2}}\right)\left({1-\frac{3}{t+2}}\right)\\
& < \frac{C}{(t+3)^2} .
\end{align*}

\end{proof}

\section{Derivation of Previous Fast Rates Results and Extensions}\label{sec:prevresults}

\subsection{Deriving the Linear Rate of Polayk \& Levitin}
Polyak \& Levitin considered in \cite{Polyak} the case in which the feasible set is strongly convex, the objective function is smooth and there exists a constant $g > 0$ such that 
\begin{eqnarray}\label{ie:polyak}
\forall{x\in\mK}: \quad \Vert{\nabla{}f(x)}\Vert_* \geq g .
\end{eqnarray}
They showed that under the lower-bounded gradient assumption, Algorithm \ref{alg:condgrad} converges with a linear rate, that is $e^{-\Theta(t)}$. 
Clearly by plugging Eq. \eqref{ie:polyak} into Lemma \ref{lem:fwineq} we have that on each iteration $t$
\begin{eqnarray*}
h_{t+1} \leq h_t\cdot\max\lbrace{\frac{1}{2}, 1- \frac{\alpha_kg}{8\beta_f}}\rbrace .
\end{eqnarray*}
which results in the same exponentially fast convergence rate as in \cite{Polyak} and following works such as \cite{Demyanov70, Dunn79}.

\subsection{Deriving a Linear Rate for Arbitrary Convex Sets in case $x^*$ is in the Interior of the Set}
Assume now that the feasible set $\mK$ is convex but not necessarily strongly convex. We assume that the objective function $f$ is smooth, convex, satisfies Eq. \eqref{ie:largegrad} with some constant $\alpha_f$ and admits a minimizer (not necessarily unique) $x^*$ that lies in the interior of $\mK$, i.e. there exists a parameter $r>0$ such that the ball of radius $r$ with respect to norm $\Vert\cdot\Vert$ centered at $x^*$ is fully contained in $\mK$ \footnote{We assume here that $\mK$ is full-dimensional. In any other case, we can assume instead that the intersection of the ball centered at $x^*$ with the smallest subspace containing $\mK$ is fully contained in $\mK$. In this case we also need to replace the gradient $\nabla{}f(x)$ with its projection onto this subspace, see also footnote 1.}. Gu{\'{e}}Lat and Marcotte \cite{GueLat1986} showed the under the above conditions, the Frank-Wolfe algorithm converges with a linear rate. We now show how a slight modification in the proof of Lemma \ref{lem:fwineq} yields this linear convergence result.

Let $w_t$ be as in the proof of Lemma \ref{lem:fwineq}, that is $w_t \in \arg\min_{w\in\vecspace, \Vert{w}\Vert \leq 1}w\cdot{}\nabla{}f(x_t)$. Instead of defining the point $\tilde{p}_t$ as in the proof of Lemma \ref{lem:fwineq} we define it to be $\tilde{p}_t = x^* + rw_t$. Because of our assumption on the location of $x^*$, it holds that $\tilde{p}_t\in\mK$. We thus have that
\begin{align*}
(\tilde{p}_t-x_t)\cdot\nabla{}f(x_t) &= (x^*_t-x_t)\cdot\nabla{}f(x_t) + rw_t\cdot\nabla{}f(x_t) \\
&\leq   - r\Vert{\nabla{}f(x_t)}\Vert_*  .
\end{align*}
Plugging this into Eq. \eqref{ie:errordec2} we have
\begin{eqnarray*}
h_{t+1} &\leq & h_t - \eta_tr\Vert{\nabla{}f(x_t)}\Vert_* + \frac{\beta_f\eta_t^2D_{\mK}^2}{2} \\
& \leq & h_t - \eta_tr\sqrt{\frac{\alpha_f}{2}}\sqrt{h_t} + \frac{\beta_f\eta_t^2D_{\mK}^2}{2} .
\end{eqnarray*}
where $D_{\mK}$ denotes the diameter of $\mK$ with respect to norm $\Vert\cdot\Vert$ and the second inequality follows from Eq. \eqref{ie:largegrad}.
By the optimal choice of $\eta_t$, we can set $\eta_t = \frac{r\sqrt{\alpha_f}\sqrt{h_t}}{\sqrt{2}\beta_fD_{\mK}^2}$ and get
\begin{eqnarray*}
h_{t+1} \leq h_t - \frac{r^2\alpha_f}{4\beta_fD_{\mK}^2}h_t ,
\end{eqnarray*}
which results in a linear convergence result.

\subsection{Relaxing the Strong Convexity of $f$}\label{sec:relax}
So far we have considered the case in which the objective function $f$ is strongly convex. Notice however that our main instrument for proving the accelerated convergence rate, i.e. Lemma \ref{lem:fwineq}, did not rely directly on strong convexity of $f$, but on the magnitude of the gradient, $\Vert{\nabla{}f(x_t)}\Vert_*$. We have seen in Eq. \eqref{ie:largegrad} that indeed if $f$ is strongly convex than the gradient is at least of the order of $\sqrt{h_t}$. We now show that there exists functions which are not strongly convex but still satisfy Eq. \eqref{ie:largegrad} and hence our results apply also for them. 

Consider the function
\begin{eqnarray*}
f(x) = \frac{1}{2}\Vert{Ax - b}\Vert_2^2 .
\end{eqnarray*}
where $x\in\Rn$, $A\in\Rmn$, $b\in\mathbb{R}^m$. Assume that $m<n$ and all rows of $A$ are linearly independent. In this case the optimization problem $\min_{x\in\mK}f(x)$ is the problem of finding a point in $\mK$ that best satisfies an under-determined linear system in terms of the mean square error. An application of the Frank-Wolfe method to this problem was studied in \cite{BeckTaboule}.
Under these assumptions, the function $f$ is smooth and convex but not strongly convex since the Hessian matrix given by $A^{\top}A$ is not positive definite (note however that the matrix $AA^{\top}$ is positive definite). 

The gradient of $f$ is given by 
\begin{eqnarray*}
\nabla{}f(x) = A^{\top}(Ax-b) .
\end{eqnarray*}
Thus we have that
\begin{align*}
\Vert{\nabla{}f(x)}\Vert_2^2 &= [A^{\top}(Ax-b)]^{\top}[A^{\top}(Ax-b)] \\
&\geq  \lambda_{\min}(AA^{\top})\Vert{Ax-b}\Vert_2^2 \\
&\geq   2\lambda_{\min}(AA^{\top})\Big(\frac{1}{2}\Vert{Ax-b}\Vert_2^2 \\
&- \frac{1}{2}\Vert{Ax^*-b}\Vert_2^2\Big) ,
\end{align*}
where $\lambda_{\min}(AA^{\top})$ denotes the smallest eigenvalue of $AA^{\top}$.
Since $AA^{\top}$ is positive definite, $\lambda_{\min}(AA^{\top}) > 0$ and it follows that $f$ satisfies Eq. \eqref{ie:largegrad}. 

Combining the result of this subsection with the previous one yields the linear convergence rate of the Frank-Wolfe method applied to the convex feasibility problem studied in \cite{BeckTaboule}.

\section{Examples of Strongly Convex Sets}\label{sec:sets}

In this section we explore convex sets for which Theorem \ref{thr:newfw} is applicable. That is, convex sets which on one hand are strongly convex, and on the other, admit a simple and efficient implementation of a linear optimization oracle. We show that various norms that give rise to natural regularization functions in machine learning, induce convex sets that fit both of the above requirements. A summary of our findings is given in Table \ref{table:sets}. We note that in all cases in which the norm parameter $p$ is smaller than $2$ (or one of the parameters $s,p$ in case of group norms), we are not aware of a practical algorithm for computing the projection.

\begin{table*}\label{table:sets}
\begin{center}
  \begin{tabular}{| c | c | c | c | c |}
    \hline
    $\vecspace$ & Domain name & Domain expression & S.C. parameter & Complexity of lin. opt.\\ \hline
    $\Rn$ & $\ell_p$ ball,  $p\in(1,2]$ & $\lbrace{x\in\Rn \, | \, \Vert{x}\Vert_p \leq r}\rbrace$ & $\frac{p-1}{r}$ & $O(\nnz)$ \\ \hline
    $\Rmn$ & Schatten $\ell_p$ ball, $p\in(1,2]$ & $\lbrace{X\in\Rmn \, | \, \Vert{\sigma(X)}\Vert_p \leq r}\rbrace$ & $\frac{p-1}{r}$ & $O(n^3)$ (SVD) \\ \hline
    $\Rmn$ & Group $\ell_{s,p}$ ball, $s,p\in(1,2]$ & $\lbrace{X\in\Rmn \, | \, \Vert{X}\Vert_{s,p} \leq r}\rbrace$& $\frac{(s-1)(p-1)}{(s+p-2)r}$ & $O(\nnz)$ \\ \hline
  \end{tabular}
  \caption{Examples of strongly convex sets with corresponding strong convexity parameter and complexity of a linear optimization oracle implementation . $\nnz$ denotes the number of non-zero entries in the linear objective and $\sigma(X)$ denotes the vector of singular values.}
\end{center}
\end{table*}

\subsection{Partial Characterization of Strongly Convex Sets}

The following lemma is taken from \cite{Nesterov10} (Theorem 12).
\begin{lemma}
Let $\vecspace$ be a finite vector space and let $f:\vecspace\rightarrow\mathbb{R}$ be non-negative, $\alpha$-strongly convex and $\beta$-smooth. Then the set $\mK = \lbrace{x \, | \, f(x) \leq r}\rbrace$ is $\frac{\alpha}{\sqrt{2\beta{}r}}$-strongly convex.
\end{lemma}

This lemma for instance shows that the Euclidean ball of radius $r$ is $1/r$-strongly convex (by applying the lemma with $f=\Vert{x}\Vert_2^2$).

The following lemma will be useful to prove that convex sets that are induced by certain norms, which do not correspond to a smooth function as in the previous lemma, are strongly convex. The proof is given in the appendix.

\begin{lemma}\label{lem:strongconvexset}
Let $\vecspace$ be a finite vector space, let $\Vert{\cdot}\Vert$ be a norm over $\vecspace$ and assume that the function $f(x) = \Vert{x}\Vert^2$ is $\alpha$-strongly convex over $\vecspace$ with respect to the norm $\Vert{\cdot}\Vert$. Then for any $r > 0$, the set $\ball_{\Vert\cdot\Vert}(r) = \lbrace{x\in\vecspace \, | \, \Vert{x}\Vert \leq r}\rbrace$ is $\frac{\alpha}{2r}$-strongly convex with respect to $\Vert\cdot\Vert$.
\end{lemma}

\subsection{$\ell_p$ Balls for $p\in(1,2]$}

Given a parameter $p\geq 1$, consider the $\ell_p$ ball of radius $r$,
\begin{eqnarray*}
\ball_p(r) = \lbrace{x\in\mathbb{R}^n \, | \, \Vert{x}\Vert_p \leq r}\rbrace .
\end{eqnarray*}

The following lemma is proved in \cite{ShayShalev07}.

\begin{lemma}\label{lem:pnorm_sc}
Fix $p\in(1,2]$. The function $\frac{1}{2}\Vert{x}\Vert_p^2$ is $(p-1)$-strongly convex w.r.t. the norm $\Vert{\cdot}\Vert_p$.
\end{lemma}

The following corollary is a consequence of combining Lemma \ref{lem:pnorm_sc} and Lemma \ref{lem:strongconvexset}. The proof is given in the appendix
\begin{corollary}\label{cor:Lpball_sc}
Fix $p\in(1,2]$. The set $\ball_p(r)$ is $\frac{p-1}{r}$-strongly convex with respect to the norm $\Vert\cdot\Vert_p$ and $\frac{(p-1)n^{\frac{1}{2}-\frac{1}{p}}}{r}$-strongly convex with respect to the norm $\Vert\cdot\Vert_2$.
\end{corollary}

The following lemma establishes that linear optimization over $\ball_p(r)$ admits a simple closed-form solution that can be computed in time that is linear in the number of non-zeros in the linear objective. The proof is given in the appendix.

\begin{lemma}\label{lem:pball_linopt}
Fix $p\in(1,2]$, $r>0$ and a linear objective $c\in\mathbb{R}^n$. Let $x\in\mathbb{R}^n$ such that $x_i = -\frac{r}{\Vert{c}\Vert_q^{q-1}}\textrm{sgn}(c_i)\vert{c_i}\vert^{q-1}$ where $q$ satisfies: $1/q + 1/p =1$, and $\textrm{sgn}(\cdot)$ is the sign function. Then $x=\arg\min_{y\in\ball_p(r)}y\cdot{}c$
\end{lemma}

\subsection{Schatten $\ell_p$ Balls for $p\in(1,2]$}

Given a matrix $X\in\mathbb{R}^{m\times{n}}$ we denote by $\sigma(X)$ the vector of singular values of $X$ in descending order, that is $\sigma(X)_1 \geq \sigma(X)_2 \geq ... \sigma(X)_{\min(m,n)}$. The Schatten $\ell_p$ norm is given by
\begin{eqnarray*}
\Vert{X}\Vert_{S(p)} = \Vert{\sigma(X)}\Vert_p = \left({\sum_{i=1}^{\min(m,n)}\sigma(X)_i^p}\right)^{1/p} .
\end{eqnarray*}

Consider the Schatten $\ell_p$ ball of radius $r$,
\begin{eqnarray*}
\ball_{S(p)}(r) = \lbrace{X\in\mathbb{R}^{m\times{}n} \, | \, \Vert{X}\Vert_{S(p)} \leq r}\rbrace .
\end{eqnarray*}

The following lemma is taken from \cite{Kakade12}.

\begin{lemma}\label{lem:schattenpnromsc}
Fix $p\in(1,2]$. The function $\frac{1}{2}\Vert{X}\Vert_{S(p)}^2$ is $(p-1)$-strongly convex w.r.t. the norm $\Vert{\cdot}\Vert_{S(p)}$.
\end{lemma}

The proof of the following corollary follows the exact same lines as the proof of Corollary \ref{cor:Lpball_sc} by using Lemma \ref{lem:schattenpnromsc} instead of Lemma \ref{lem:pnorm_sc}.

\begin{corollary}\label{cor:Spball_sc}
Fix $p\in(1,2]$. The set $\ball_{S(p)}(r)$ is $\frac{p-1}{r}$-strongly convex with respect to the norm $\Vert\cdot\Vert_{S(p)}$ and $\frac{(p-1)\min(m,n)^{\frac{1}{2}-\frac{1}{p}}}{r}$-strongly convex with respect to the frobenius norm $\Vert\cdot\Vert_F$.
\end{corollary}

The following lemma establishes that linear optimization over $\ball_{S(p)}(r)$ admits a simple closed-form solution given the \textit{singular value decomposition} of the linear objective. The proof is given in the appendix.

\begin{lemma}\label{lem:sball_linopt}
Fix $p\in(1,2]$, $r>0$ and a linear objective $C\in\mathbb{R}^{m\times{n}}$. Let $C=U\Sigma{}V^{\top}$ be the \textit{singular value decomposition} of $C$. Let $\sigma$ be a vector such that $\sigma_i = -\frac{r}{\Vert{\sigma(C)}\Vert_q^{q-1}}\sigma(C)_i^{q-1}$ where $q$ satisfies: $1/q + 1/p =1$. Finally, let $X=UDiag(\sigma)V^{\top}$ where $Diag(\sigma)$ is an $m\times{n}$ diagonal matrix with the vector $\sigma$ as the main diagonal. Then $X=\arg\min_{Y\in\ball_{S(p)}(r)}Y\bullet{}C$, where $\bullet$ denotes the standard inner product for matrices.
\end{lemma}

\subsection{Group $\ell_{s,p}$ Balls for $s,p\in(1,2]$}

Given a matrix $X\in\mathbb{R}^{m\times{n}}$ denote by $X_i\in\mathbb{R}^n$ the $i$th row of $X$. That is $X = (X_1, X_2, ...,X_m)^{\top}$.

The $\ell_{s,p}$ norm of $X$ is given by,
\begin{eqnarray*}
\Vert{X}\Vert_{s,p}=\Vert{(\Vert{X_1}\Vert_s, \Vert{X_2}\Vert_s,..., \Vert{X_m}\Vert_s)}\Vert_p .
\end{eqnarray*}

We define the  $\ell_{s,p}$ ball as follows:
\begin{eqnarray*}
\ball_{s,p}(r) = \lbrace{X\in\mathbb{R}^{m\times{}n} \, | \, \Vert{X}\Vert_{s,p} \leq r}\rbrace .
\end{eqnarray*}

The proof of the following lemma is given in the appendix.
\begin{lemma}\label{lem:groupballsc}
Let $s,p\in(1,2]$. The set $\ball_{s,p}(r)$ is $\frac{(s-1)(p-1)}{(s+p-2)r}$-strongly convex with respect to the norm $\Vert\cdot\Vert_{s,p}$ and $n^{\frac{1}{s}-\frac{1}{2}}m^{\frac{1}{p}-\frac{1}{2}}\frac{(s-1)(p-1)}{(s+p-2)r}$-strongly convex with respect to the frobenius norm $\Vert\cdot\Vert_F$.
\end{lemma}

The following lemma establishes that linear optimization over $\ball_{s,p}(r)$ admits a simple closed-form solution that can be computed in time that is linear in the number of non-zeros in the linear objective. The proof is given in the appendix.

\begin{lemma}\label{lem:groupball_linopt}
Fix $s,p\in(1,2]$, $r>0$ and a linear objective $C\in\mathbb{R}^{m\times{n}}$.  Let $X\in\mathbb{R}^{m\times{n}}$ be such that $X_{i,j} = -\frac{r}{\Vert{C}\Vert_{z,q}^{q-1}\Vert{C_i}\Vert_z^{z-q}}\textrm{sgn}(C_{i,j})\vert{C_{i,j}}\vert^{z-1}$ where $z$ satisfies: $1/s+1/z=1$, $q$ satisfies: $1/p+1/q=1$ and $C_i$ denotes the $i$th row of $C$. Then $X=\arg\min_{Y\in\ball_{s,p}(r)}Y\bullet{}C$, where $\bullet$ denotes the standard inner product for matrices.
\end{lemma}

\section{Conclusions and Open Problems}
In this paper we proved that the Frank-Wolfe algorithm converges at an accelerated rate of $O(1/t^2)$ for smooth and strongly-convex optimization over strongly-convex sets, beating the known tight convergence rate of the method  for general smooth and convex optimization. This is one of the very few known results that achieve such an improvement in convergence rate under natural and standard assumptions (i.e. strong convexity etc.). We have further demonstrated that various regularization functions in machine learning give rise to strongly convex sets. We have also demonstrated how previous fast convergence rate results follow easily from our analysis.

The following questions naturally arise.

It is known that in case the objective function is both smooth and strongly convex, projection/prox-based methods achieve a convergence rate of $O(\log(1/\epsilon))$. Is it possible to achieve such a rate for the FW method in case the convex set is strongly convex?

We have shown that it is possible to obtain faster rates for optimization over balls induced by norms that give rise to strongly convex functions. Is it possible to obtain faster rates for balls induced by norms that do not give rise to strongly convex functions (but rather to smooth functions)? e.g. is it possible to obtain faster rates for $\ell_p$ balls for $p > 2$.

Finally, the most intriguing question is to give a linear optimization oracle-based method that enjoys the same convergence rate, at least in terms of the approximation error, as optimal projection/prox-based gradient methods, in any regime (including non-smooth problems).  A progress in this direction was made recently by Garber and Hazan \cite{Garber13b} who showed that in case the feasible set is a polytope, a variant of the FW-method obtains the same rates as the projected (sub)gradient descent method.

\section*{Acknowledgments} 
The research leading to these results has received funding from the European Union’s Seventh Framework Programme (FP7/2007-2013) under grant agreement n$^\circ$ 336078 -- ERC-SUBLRN.

\bibliography{bib}

\begin{thebibliography}{25}
\providecommand{\natexlab}[1]{#1}
\providecommand{\url}[1]{\texttt{#1}}
\expandafter\ifx\csname urlstyle\endcsname\relax
  \providecommand{\doi}[1]{doi: #1}\else
  \providecommand{\doi}{doi: \begingroup \urlstyle{rm}\Url}\fi

\bibitem[Ahipasaoglu et~al.(2008)Ahipasaoglu, Sun, and Todd]{Ahipasaoglu08}
Ahipasaoglu, S.~Damla, Sun, Peng, and Todd, Michael~J.
\newblock Linear convergence of a modified frank-wolfe algorithm for computing
  minimum-volume enclosing ellipsoids.
\newblock \emph{Optimization Methods and Software}, 23\penalty0 (1):\penalty0
  5--19, 2008.

\bibitem[Beck \& Teboulle(2004)Beck and Teboulle]{BeckTaboule}
Beck, Amir and Teboulle, Marc.
\newblock A conditional gradient method with linear rate of convergence for
  solving convex linear systems.
\newblock \emph{Math. Meth. of OR}, 59\penalty0 (2):\penalty0 235--247, 2004.

\bibitem[Clarkson(2008)]{Clarkson}
Clarkson, Kenneth~L.
\newblock Coresets, sparse greedy approximation, and the frank-wolfe algorithm.
\newblock In \emph{Proceedings of the Nineteenth Annual {ACM-SIAM} Symposium on
  Discrete Algorithms, {SODA}}, 2008.

\bibitem[Demyanov \& Rubinov(1970)Demyanov and Rubinov]{Demyanov70}
Demyanov, Vladimir~F. and Rubinov, Aleksandr~M.
\newblock \emph{{Approximate methods in optimization problems}}.
\newblock Elsevier Publishing Company, 1970.

\bibitem[Dud\'{\i}k et~al.(2012)Dud\'{\i}k, Harchaoui, and Malick]{Dudik12a}
Dud\'{\i}k, Miroslav, Harchaoui, Za\"{\i}d, and Malick, J{\'e}r{\^o}me.
\newblock Lifted coordinate descent for learning with trace-norm
  regularization.
\newblock \emph{Journal of Machine Learning Research - Proceedings Track},
  22:\penalty0 327--336, 2012.

\bibitem[Dunn(1979)]{Dunn79}
Dunn, Joseph~C.
\newblock {Rates of Convergence for Conditional Gradient Algorithms Near
  Singular and Nonsingular Extremals}.
\newblock \emph{SIAM Journal on Control and Optimization}, 17\penalty0 (2),
  1979.

\bibitem[Frank \& Wolfe(1956)Frank and Wolfe]{FrankWolfe}
Frank, M. and Wolfe, P.
\newblock An algorithm for quadratic programming.
\newblock \emph{Naval Research Logistics Quarterly}, 3:\penalty0 149--154,
  1956.

\bibitem[Garber \& Hazan(2013{\natexlab{a}})Garber and Hazan]{Garber13}
Garber, Dan and Hazan, Elad.
\newblock Playing non-linear games with linear oracles.
\newblock In \emph{54th Annual {IEEE} Symposium on Foundations of Computer
  Science, {FOCS}}, 2013{\natexlab{a}}.

\bibitem[Garber \& Hazan(2013{\natexlab{b}})Garber and Hazan]{Garber13b}
Garber, Dan and Hazan, Elad.
\newblock A linearly convergent conditional gradient algorithm with
  applications to online and stochastic optimization.
\newblock \emph{CoRR}, abs/1301.4666, 2013{\natexlab{b}}.

\bibitem[Gu{\'{e}}Lat \& Marcotte(1986)Gu{\'{e}}Lat and Marcotte]{GueLat1986}
Gu{\'{e}}Lat, Jacques and Marcotte, Patrice.
\newblock {Some comments on Wolfe's `away step'}.
\newblock \emph{Mathematical Programming}, 35\penalty0 (1), 1986.

\bibitem[Harchaoui et~al.(2012)Harchaoui, Douze, Paulin, Dud{\'{\i}}k, and
  Malick]{Dudik12b}
Harchaoui, Za{\"{\i}}d, Douze, Matthijs, Paulin, Mattis, Dud{\'{\i}}k,
  Miroslav, and Malick, J{\'{e}}r{\^{o}}me.
\newblock Large-scale image classification with trace-norm regularization.
\newblock In \emph{{IEEE} Conference on Computer Vision and Pattern
  Recognition, {CVPR}}, 2012.

\bibitem[Hazan(2008)]{Hazan08}
Hazan, Elad.
\newblock Sparse approximate solutions to semidefinite programs.
\newblock In \emph{8th Latin American Theoretical Informatics Symposium,
  {LATIN}}, 2008.

\bibitem[Hazan \& Kale(2012)Hazan and Kale]{Hazan12}
Hazan, Elad and Kale, Satyen.
\newblock Projection-free online learning.
\newblock In \emph{Proceedings of the 29th International Conference on Machine
  Learning, {ICML}}, 2012.

\bibitem[Jaggi(2013)]{Jaggi13b}
Jaggi, Martin.
\newblock Revisiting frank-wolfe: Projection-free sparse convex optimization.
\newblock In \emph{Proceedings of the 30th International Conference on Machine
  Learning, {ICML}}, 2013.

\bibitem[Jaggi \& Sulovsk{\'{y}}(2010)Jaggi and Sulovsk{\'{y}}]{Jaggi10}
Jaggi, Martin and Sulovsk{\'{y}}, Marek.
\newblock A simple algorithm for nuclear norm regularized problems.
\newblock In \emph{Proceedings of the 27th International Conference on Machine
  Learning, {ICML}}, 2010.

\bibitem[Journ{\'e}e et~al.(2010)Journ{\'e}e, Nesterov, Richt{\'a}rik, and
  Sepulchre]{Nesterov10}
Journ{\'e}e, Michel, Nesterov, Yurii, Richt{\'a}rik, Peter, and Sepulchre,
  Rodolphe.
\newblock Generalized power method for sparse principal component analysis.
\newblock \emph{Journal of Machine Learning Research}, 11:\penalty0 517--553,
  2010.

\bibitem[Kakade et~al.(2012)Kakade, Shalev-Shwartz, and Tewari]{Kakade12}
Kakade, Sham~M., Shalev-Shwartz, Shai, and Tewari, Ambuj.
\newblock Regularization techniques for learning with matrices.
\newblock \emph{Journal of Machine Learning Research}, 13:\penalty0 1865--1890,
  2012.

\bibitem[Lacoste-Julien \& Jaggi(2013)Lacoste-Julien and Jaggi]{Jaggi13c}
Lacoste-Julien, Simon and Jaggi, Martin.
\newblock An affine invariant linear convergence analysis for frank-wolfe
  algorithms.
\newblock \emph{CoRR}, abs/1312.7864, 2013.

\bibitem[Lacoste{-}Julien et~al.(2013)Lacoste{-}Julien, Jaggi, Schmidt, and
  Pletscher]{Jaggi13a}
Lacoste{-}Julien, Simon, Jaggi, Martin, Schmidt, Mark~W., and Pletscher,
  Patrick.
\newblock Block-coordinate frank-wolfe optimization for structural svms.
\newblock In \emph{Proceedings of the 30th International Conference on Machine
  Learning, {ICML}}, 2013.

\bibitem[Lan(2013)]{Lan13}
Lan, Guanghui.
\newblock The complexity of large-scale convex programming under a linear
  optimization oracle.
\newblock \emph{CoRR}, abs/1309.5550, 2013.

\bibitem[Laue(2012)]{Laue12}
Laue, S{\"{o}}ren.
\newblock A hybrid algorithm for convex semidefinite optimization.
\newblock In \emph{Proceedings of the 29th International Conference on Machine
  Learning, {ICML}}, 2012.

\bibitem[Levitin \& Polyak(1966)Levitin and Polyak]{Polyak}
Levitin, Evgeny~S and Polyak, Boris~T.
\newblock Constrained minimization methods.
\newblock \emph{USSR Computational mathematics and mathematical physics},
  6:\penalty0 1--50, 1966.

\bibitem[Migdalas(1994)]{Migdalas}
Migdalas, Athanasios.
\newblock A regularization of the frank—wolfe method and unification of
  certain nonlinear programming methods.
\newblock \emph{Mathematical Programming}, 65:\penalty0 331--345, 1994.

\bibitem[Shalev{-}Shwartz et~al.(2011)Shalev{-}Shwartz, Gonen, and
  Shamir]{ShalevShwartz11}
Shalev{-}Shwartz, Shai, Gonen, Alon, and Shamir, Ohad.
\newblock Large-scale convex minimization with a low-rank constraint.
\newblock In \emph{Proceedings of the 28th International Conference on Machine
  Learning, {ICML}}, 2011.

\bibitem[Shwartz(2007)]{ShayShalev07}
Shwartz, Shay~Sahlev.
\newblock Phd thesis.
\newblock 2007.

\end{thebibliography}
\bibliographystyle{icml2015}

\newpage

\appendix

\section{Proof of Theorem \ref{thr:originalfw}}
\begin{proof}
Fix an iteration $t$. By the $\beta_f$-smoothness of $f$ we have that
\begin{eqnarray}\label{eq:oldfwproof}
h_{t+1} &=& f(x_t + \eta_t(p_t-x_t)) - f(x^*) \nonumber \\
&\leq & f(x_t) - f(x^*) + \eta_t(p_t-x_t)\cdot\nabla{}f(x_t)\nonumber  \\
&+& \frac{\eta_t^2\beta_f}{2}\Vert{p_t-x_t}\Vert^2 \nonumber \\
&\leq & h_t - \eta_th_t + \frac{\eta_t^2\beta_fD_{\mK}^2}{2} ,
\end{eqnarray}
where the last inequality follows from convexity of $f$. Notice that by the optimal choice of $\eta_t$ in Algorithm \ref{alg:condgrad}, it holds in particular that $h_{t+1} \leq h_t$ (by setting $\eta_t=0$ in Eq. \eqref{eq:oldfwproof}).

Fix $C = 8\beta_fD_{\mK}^2$. We now prove by induction on $t$ that $h_t \leq \frac{C}{t}$. 

For the base case $t=1$ we notice that by the optimal choice of $\eta_0$ in Algorithm \ref{alg:condgrad} we can in particular set $\eta_0 = 1$ and thus it follows from Eq. \eqref{eq:oldfwproof} that $h_1 \leq \frac{\beta_fD_{\mK}^2}{2} < C$ as needed.

Assume now that the induction holds for $t\geq 1$. That is $h_t \leq \frac{C}{t}$. We consider two cases.

If $h_t \leq \frac{C}{2t}$ then we have 
\begin{eqnarray*}
h_{t+1} \leq h_t \leq \frac{C}{2t} = \frac{C}{t+1}\cdot\frac{t+1}{2t} \leq \frac{C}{t+1} ,
\end{eqnarray*}
where the last inequality holds for any $t\geq 1$.

Otherwise it holds that $h_t > \frac{C}{2t}$. Using Eq. \eqref{eq:oldfwproof} again we have
\begin{eqnarray*}
h_{t+1} \leq h_t -\eta_th_t + \frac{\eta_t^2\beta_fD_{\mK}^2}{2} .
\end{eqnarray*}
By the optimal choice of $\eta_t$ in Algorithm \ref{alg:condgrad} we can set $\eta_t = \frac{h_t}{\beta_fD_{\mK}^2}$ and get 
\begin{eqnarray*}
h_{t+1} &\leq& h_t - \frac{1}{2\beta_fD_{\mK}^2}h_t^2 < \frac{C}{t} - \frac{C^2}{8\beta_fD_{\mK}^2t^2} \\
&=& \frac{C}{t+1}\left({\frac{t+1}{t} - \frac{C(t+1)}{8\beta_fD_{\mK}^2t^2}}\right) \\
&< & \frac{C}{t+1}\left({1 + \frac{1}{t} - \frac{Ct}{8\beta_fD_{\mK}^2t^2}}\right) .
\end{eqnarray*}
Thus for $C\geq 8\beta_fD_{\mK}^2$ we have that $h_{t+1} \leq \frac{C}{t+1}$.
\end{proof}

\section{Proofs of Lemmas and Corollaries from Section \ref{sec:sets}}

%\subsection{proof of lemma \ref{lem:strongconvexset}}

\subsection{Proof of Lemma \ref{lem:strongconvexset}}
\begin{proof}
It suffices to show that given $x,y\in\vecspace$ such that $f(x)\leq r^2, f(y) \leq r^2$, a scalar $\gamma\in[0,1]$ and $z\in\vecspace$ such that $\Vert{z}\Vert \leq \frac{\alpha}{4r}\gamma(1-\gamma)\Vert{x-y}\Vert^2$, it holds that, $f(\gamma{}x+(1-\gamma)y + z) \leq r^2$.

By the definition of $f$ and the triangle inequality for $\Vert\cdot\Vert$ we have
\begin{eqnarray}\label{ie1}
&& f(\gamma{}x+(1-\gamma)y + z) = \Vert{\gamma{}x+(1-\gamma)y + z}\Vert^2 \leq \nonumber \\
&&  \left({\Vert{\gamma{}x+(1-\gamma)y}\Vert + \Vert{z}\Vert}\right)^2 = \nonumber \\
&& \left({\sqrt{f(\gamma{}x+(1-\gamma)y)} + \Vert{z}\Vert}\right)^2 .
\end{eqnarray}

Since $f$ is $\alpha$ strongly convex with respect to $\Vert\cdot\Vert$ we have that
\begin{eqnarray*}
&&f(\gamma{}x+(1-\gamma)y)  \leq \\
&& \gamma{}f(x) + (1-\gamma)f(y) - \frac{\alpha}{2}\gamma(1-\gamma)\Vert{x-y}\Vert^2 \leq \\
&&  r^2 - \frac{\alpha}{2}\gamma(1-\gamma)\Vert{x-y}\Vert^2 .	
\end{eqnarray*}

The function $g(t)=\sqrt{t}$ is concave, meaning $\sqrt{a-b} = g(a-b) \leq g(a) - g'(a)\cdot{}b = \sqrt{a}-\frac{b}{2\sqrt{a}}$. Thus, 
\begin{eqnarray*}
\sqrt{f(\gamma{}x+(1-\gamma)y)}
&\leq & \sqrt{r^2 - \frac{\alpha}{2}\gamma(1-\gamma)\Vert{x-y}\Vert^2} \\
& \leq & r - \frac{\alpha\gamma(1-\gamma)\Vert{x-y}\Vert^2}{4r} .
\end{eqnarray*}
Plugging back in Eq. \eqref{ie1} we have
\begin{eqnarray*}
&&f(\gamma{}x+(1-\gamma)y + z) \leq \\
&& \left({r - \frac{\alpha\gamma(1-\gamma)\Vert{x-y}\Vert^2}{4r} + \Vert{z}\Vert}\right)^2 .
\end{eqnarray*}

By our assumption on $\Vert{z}\Vert$ we have
\begin{eqnarray*}
f(\gamma{}x+(1-\gamma)y + z) &\leq &   \Big(r - \frac{\alpha\gamma(1-\gamma)\Vert{x-y}\Vert^2}{4r}\\
&+&\frac{\alpha}{4r}\gamma(1-\gamma)\Vert{x-y}\Vert^2\Big)^2 \\
&=& r^2 .
\end{eqnarray*}
\end{proof}

\subsection{Proof of Corollary \ref{cor:Lpball_sc}}

\begin{proof}
The strong convexity of the set w.r.t. $\Vert\cdot\Vert_p$ is an immediate consequence of Lemma \ref{lem:strongconvexset}.

Since $\ball_p(r)$ is $\alpha=(p-1)/r$ strongly convex w.r.t. the norm $\Vert\cdot\Vert_p$, we have that given $x,y\in\ball_p(r)$, $\gamma\in[0,1]$ and $z\in\mathbb{R}^n$ such that $\Vert{z}\Vert_p \leq 1$ it holds that
\begin{eqnarray*}
\gamma{}x + (1-\gamma)y + \frac{\alpha}{2}\gamma(1-\gamma)\Vert{x-y}\Vert_p^2z\in\ball_p(r) .
\end{eqnarray*}

For any $p\in(1,2]$ and vector $v\in\mathbb{R}^n$ it holds that
\begin{eqnarray}\label{ie:pnorm}
\Vert{v}\Vert_2 \leq \Vert{v}\Vert_p \leq n^{\frac{1}{p} - \frac{1}{2}}\Vert{v}\Vert_2 .
\end{eqnarray}

Given a vector $z'\in\mathbb{R}^n$ such that $\Vert{z'}\Vert_F \leq 1$ we have that
\begin{eqnarray*}
\Vert{\frac{\alpha}{2}\gamma(1-\gamma)\Vert{x-y}\Vert_2^2z'}\Vert_p = \\
 \frac{\alpha}{2}\gamma(1-\gamma)\Vert{x-y}\Vert_2^2\Vert{z'}\Vert_p .
\end{eqnarray*}
Using Eq. \eqref{ie:pnorm} we have
\begin{eqnarray*}
&&\Vert{\frac{\alpha}{2}\gamma(1-\gamma)\Vert{x-y}\Vert_2^2z'}\Vert_p  \leq \\
&& \frac{\alpha}{2}\gamma(1-\gamma)\Vert{x-y}\Vert_p^2n^{\frac{1}{p}-\frac{1}{2}}\Vert{z'}\Vert_2 \leq \\
&& \frac{\alpha{}n^{\frac{1}{p}-\frac{1}{2}}}{2}\gamma(1-\gamma)\Vert{x-y}\Vert_p^2 .
\end{eqnarray*}

Hence, $\ball_p(r)$ is $\alpha{}n^{\frac{1}{2}-\frac{1}{p}} = \frac{(p-1)n^{\frac{1}{2}-\frac{1}{p}}}{r}$-strongly convex with respect to $\Vert\cdot\Vert_2$. 
\end{proof}

\subsection{Proof of Lemma \ref{lem:pball_linopt}}

\begin{proof}
Since $\Vert{\cdot}\Vert_p$ and $\Vert{\cdot}\Vert_q$ are dual norms, we have using Holder's inequality that for all $x\in\ball_p(r)$,
\begin{eqnarray*}
x\cdot c \geq -\Vert{x}\Vert_p\Vert{c}\Vert_q \geq -r\Vert{c}\Vert_q .
\end{eqnarray*}

Thus choosing $x_i = -\frac{r}{\Vert{c}\Vert_q^{q-1}}\textrm{sgn}(c_i)\vert{c_i}\vert^{q-1}$ we have that
\begin{eqnarray*}
x\cdot c &=& -\sum_{i=1}^n\frac{r}{\Vert{c}\Vert_q^{q-1}}\textrm{sgn}(c_i)\vert{c_i}\vert^{q-1}\cdot{}c_i \\
&=& -\sum_{i=1}^n\frac{r}{\Vert{c}\Vert_q^{q-1}}\vert{c_i}\vert^{q} 
= -\frac{r}{\Vert{c}\Vert_q^{q-1}}\Vert{c}\Vert_q^q \\
&=& -r\Vert{c}\Vert_q .
\end{eqnarray*}

Moreover,
\begin{eqnarray*}
\Vert{x}\Vert_p^p = \frac{r^p}{\left({\Vert{c}\Vert_q^{q-1}}\right)^p}\sum_{i=1}^n\left({\vert{c_i}\vert^{q-1}}\right)^p .
\end{eqnarray*}
Since $p = q/(q-1)$ we have that
\begin{eqnarray*}
\Vert{x}\Vert_p^p = \frac{r^p}{\Vert{c}\Vert_q^q}\sum_{i=1}^n\vert{c_i}\vert^{q} = r^p .
\end{eqnarray*}
Thus we have that $x\in\ball_p(r)$.
\end{proof}

\subsection{Proof of Lemma \ref{lem:sball_linopt}}

\begin{proof}
Since $\Vert{\cdot}\Vert_{S(p)}$ and $\Vert{\cdot}\Vert_{S(q)}$ are dual norms we from Holder's inequality that for all $X\in\ball_{S(p)}(r)$,
\begin{eqnarray*}
X\bullet{}C \geq -\Vert{X}\Vert_{S(p)}\Vert{C}\Vert_{S(q)} \geq -r\Vert{C}\Vert_{S(q)} = r\Vert{\sigma(C)}\Vert_q .
\end{eqnarray*}

By our choice of $X$ we have that

\begin{eqnarray*}
X\bullet{}C &=& \textrm{Tr}(X^{\top}C) = \textrm{Tr}(VDiag(\sigma)^{\top}U^{\top}U\Sigma{}V^{\top}) \\
&=& \textrm{Tr}(VDiag(\sigma)^{\top}\Sigma{}V^{\top}) \\
&=& \textrm{Tr}(V^{\top}VDiag(\sigma)^{\top}\Sigma) = \textrm{Tr}(Diag(\sigma)^{\top}\Sigma) \\
&=&\sum_{i=1}^{\min(m,n)}-\frac{r}{\Vert{\sigma(C)}\Vert_q^{q-1}}\sigma(C)_i^{q-1}\cdot\sigma(C)_i \\
&=& -\frac{r}{\Vert{\sigma(C)}\Vert_q^{q-1}}\sum_{i=1}^{\min(m,n)}\sigma(C)_i^q \\
&=& -r\Vert{\sigma(C)}\Vert_q .
\end{eqnarray*}

Moreover,
\begin{eqnarray*}
\Vert{X}\Vert_{S(p)}^p =  \Vert{\sigma(X)}\Vert_p^p = \frac{r^p}{\left({\Vert{\sigma(C)}\Vert_q^{q-1}}\right)^p}\sum_{i=1}^n\left({\sigma(C)_i^{q-1}}\right)^p .
\end{eqnarray*}
Since $p = q/(q-1)$ we have that
\begin{eqnarray*}
\Vert{X}\Vert_{S(p)}^p = \frac{r^p}{\Vert{\sigma(C)}\Vert_q^q}\sum_{i=1}^n\vert{\sigma(C)_i}\vert^{q} = r^p .
\end{eqnarray*}
Thus we have that $X\in\ball_{S(p)}(r)$.
\end{proof}

\subsection{Proof of Lemma \ref{lem:groupballsc}}

The following lemma will be of use in the proof.

\begin{lemma}\label{lem:spnormineq}
for any matrix $A\in\mathbb{R}^{m\times{n}}$ and $s,p\in(1,2]$ it holds that
\begin{eqnarray*}
\Vert{A}\Vert_{F} \leq \Vert{A}\Vert_{s,p} \leq n^{\frac{1}{s} - \frac{1}{2}}m^{\frac{1}{p} - \frac{1}{2}}\Vert{A}\Vert_F .
\end{eqnarray*}
\end{lemma}

\begin{proof}
For any vector $v\in\mathbb{R}^n$ and $p\in(1,2]$ it holds that
\begin{eqnarray}\label{ie:pnorm2}
\Vert{v}\Vert_2 \leq \Vert{v}\Vert_p \leq n^{\frac{1}{p} - \frac{1}{2}}\Vert{v}\Vert_2 .
\end{eqnarray}
Denote by $A_i$ the $i$th row of $A$. For any $i\in[m]$ and $p\in(1,2]$ it holds that
\begin{eqnarray}\label{ie2}
\Vert{A_i}\Vert_2 \leq \Vert{A_i}\Vert_p \leq n^{\frac{1}{p} - \frac{1}{2}}\Vert{A_i}\Vert_2 .
\end{eqnarray}

Note that by definition $\Vert\cdot\Vert_F \equiv \Vert\cdot\Vert_{2,2}$. Applying Eq. \eqref{ie:pnorm2} and \eqref{ie2} we have,
\begin{eqnarray*}
\Vert{A}\Vert_F &=& \Vert{A}\Vert_{2,2} = \Vert{(\Vert{A_1}\Vert_2, \Vert{A_2}\Vert_2, ..., \Vert{A_m}\Vert_2)}\Vert_2 \\
&\leq &  \Vert{(\Vert{A_1}\Vert_s, \Vert{A_2}\Vert_s, ..., \Vert{A_m}\Vert_s)}\Vert_p \\
&\leq & n^{\frac{1}{s} - \frac{1}{2}}m^{\frac{1}{p} - \frac{1}{2}}\Vert{(\Vert{A_1}\Vert_2, \Vert{A_2}\Vert_2, ..., \Vert{A_m}\Vert_2)}\Vert_2 \\
&=& n^{\frac{1}{s} - \frac{1}{2}}m^{\frac{1}{p} - \frac{1}{2}}\Vert{A}\Vert_F .
\end{eqnarray*}
\end{proof}

We can now prove Lemma \ref{lem:groupballsc}.

\begin{proof}
Let $z,q$ be such that $1/z + 1/s = 1$ and $1/q + 1/p = 1$. Note that $z,q\in[2,\infty)$. The norm $\Vert\cdot\Vert_{z,q}$ is the dual norm to $\Vert\cdot\Vert_{s,p}$ (see \cite{Kakade12} for instance). 

According to Lemma \ref{lem:pnorm_sc}, the functions $\Vert{x}\Vert_s^2$  and $\Vert{x}\Vert_p^2$ are $\alpha_s = 2(s-1)$-strongly convex w.r.t. $\Vert\cdot\Vert_p$ and $\alpha_p = 2(p-1)$-strongly convex w.r.t. $\Vert\cdot\Vert_q$ respectively. Hence by the \textit{strong convexity / smoothness duality} (see Theorem 3 in \cite{Kakade12}) we have that the functions $\Vert{x}\Vert_z^2$ and $\Vert{x}\Vert_q^2$ are $\alpha_s^{-1}$-smooth w.r.t. $\Vert\cdot\Vert_z$ and $\alpha_p^{-1}$-smooth w.r.t. $\Vert\cdot\Vert_q$ respectively. 

By Theorem 13 in \cite{Kakade12} we have that the function $\Vert{X}\Vert_{z,q}^2$ is $(\alpha_p^{-1} + \alpha_s^{-1})$-smooth with respect to the norm $\Vert\cdot\Vert_{z,q}$. Again using the \textit{strong convexity / smoothness duality} we have that $\Vert{X}\Vert_{s,p}^2$ is $\left({\alpha_p^{-1} + \alpha_s^{-1}}\right)^{-1} = \frac{\alpha_p\alpha_s}{\alpha_p+\alpha_s}$ strongly convex with respect to the norm $\Vert\cdot\Vert_{s,p}$. The first part of the lemma now follows from applying Lemma \ref{lem:strongconvexset}.

Since $\ball_{s,p}(r)$ is $\alpha=\frac{(s-1)(p-1)}{(s+p-2)r}$ strongly convex w.r.t. the norm $\Vert\cdot\Vert_{s,p}$, we have that given $X,Y\in\ball_{s,p}(r)$, $\gamma\in[0,1]$ and $Z\in\mathbb{R}^{m\times{n}}$ such that $\Vert{Z}\Vert_{s,p} \leq 1$ it holds that
\begin{eqnarray*}
\gamma{}X + (1-\gamma)Y + \frac{\alpha}{2}\gamma(1-\gamma)\Vert{X-Y}\Vert_{s,p}^2Z\in\ball_{s,p}(r) .
\end{eqnarray*}

Given a matrix $Z'\in\mathbb{R}^{m\times{n}}$ such that $\Vert{Z'}\Vert_{F} \leq 1$ we have that
\begin{eqnarray*}
\Vert{\frac{\alpha}{2}\gamma(1-\gamma)\Vert{X-Y}\Vert_F^2Z'}\Vert_{s,p} = \\
\frac{\alpha}{2}\gamma(1-\gamma)\Vert{x-y}\Vert_F^2\Vert{Z'}\Vert_{s,p}  .
\end{eqnarray*}
Using Lemma \ref{lem:spnormineq} we have
\begin{eqnarray*}
&&\Vert{\frac{\alpha}{2}\gamma(1-\gamma)\Vert{X-Y}\Vert_F^2Z'}\Vert_{s,p} \leq \\
&& \frac{\alpha}{2}\gamma(1-\gamma)\Vert{X-Y}\Vert_{s,p}^2n^{\frac{1}{s}-\frac{1}{2}}m^{\frac{1}{p}-\frac{1}{2}}\Vert{Z'}\Vert_F  \leq \\
&& \frac{\alpha{}n^{\frac{1}{s}-\frac{1}{2}}m^{\frac{1}{p}-\frac{1}{2}}}{2}\gamma(1-\gamma)\Vert{X-Y}\Vert_{s,p}^2 .
\end{eqnarray*}

Hence, $\ball_{s,p}(r)$ is $\alpha{}n^{\frac{1}{s}-\frac{1}{2}}m^{\frac{1}{p}-\frac{1}{2}} = n^{\frac{1}{s}-\frac{1}{2}}m^{\frac{1}{p}-\frac{1}{2}}\frac{(s-1)(p-1)}{(s+p-2)r}$ strongly convex with respect to $\Vert\cdot\Vert_F$. 
\end{proof}

\subsection{Proof of Lemma \ref{lem:groupball_linopt}}

\begin{proof}
Since by choice of $z,q$ it holds that $\Vert\cdot\Vert_{s,p}, \Vert\cdot\Vert_{z,q}$ are dual norms, we have by Holder's inequality that
\begin{eqnarray*}
X\bullet{C} \geq - \Vert{X}\Vert_{s,p}\Vert{C}\Vert_{z,q} \geq -r\Vert{C}\Vert_{z,q} .
\end{eqnarray*}

Thus, choosing $$X_{i,j} = -\frac{r}{\Vert{C}\Vert_{z,q}^{q-1}\Vert{C_i}\Vert_z^{z-q}}\textrm{sgn}(C_{i,j})\vert{C_{i,j}}\vert^{z-1},$$ we have that
\begin{eqnarray*}
&& X\bullet{}C = \sum_{i\in[m],j\in[n]}X_{i,j}C_{i,j} =\\
&& \sum_{i\in[m],j\in[n]}-\frac{r}{\Vert{C}\Vert_{z,q}^{q-1}\Vert{C_i}\Vert_z^{z-q}}\textrm{sgn}(C_{i,j})\vert{C_{i,j}}\vert^{z-1}\cdot{}C_{i,j} = \\
&& \sum_{i\in[m],j\in[n]}-\frac{r}{\Vert{C}\Vert_{z,q}^{q-1}\Vert{C_i}\Vert_z^{z-q}}\vert{C_{i,j}}\vert^{z} = \\
&& \sum_{i\in[m]}-\frac{r}{\Vert{C}\Vert_{z,q}^{q-1}\Vert{C_i}\Vert_z^{z-q}}\sum_{j\in[n]}\vert{C_{i,j}}\vert^z = \\
&& \sum_{i\in[m]}-\frac{r}{\Vert{C}\Vert_{z,q}^{q-1}\Vert{C_i}\Vert_z^{z-q}}\Vert{C_i}\Vert_z^z =  \sum_{i\in[m]}-\frac{r}{\Vert{C}\Vert_{z,q}^{q-1}}\Vert{C_i}\Vert_z^q = \\
&& -\frac{r}{\Vert{C}\Vert_{z,q}^{q-1}}\sum_{i\in[m]}\Vert{C_i}\Vert_z^q = -\frac{r}{\Vert{C}\Vert_{z,q}^{q-1}}\Vert{C}\Vert_{z,q}^q = -r\Vert{C}\Vert_{z,q} .
\end{eqnarray*}

Moreover, for all $i\in[m]$ it holds that

\begin{eqnarray*}
\Vert{X_i}\Vert_{s}^s = \sum_{j=1}^n\vert{X_{i,j}}\vert^s = \frac{r^s}{\Vert{C}\Vert_{z,q}^{s(q-1)}\Vert{C_i}\Vert_z^{s(z-q)}}\sum_{i=j}^n\vert{C_{i,j}}\vert^{s(z-1)} .
\end{eqnarray*}
Since $s=z/(z-1)$ we have
\begin{eqnarray*}
\Vert{X_i}\Vert_{s}^s = \frac{r^s}{\Vert{C}\Vert_{z,q}^{s(q-1)}\Vert{C_i}\Vert_z^{s(z-q)}}\Vert{C_i}\Vert_z^z =  \frac{\Vert{C_i}\Vert_z^{sq - z(s-1)}}{\Vert{C}\Vert_{z,q}^{s(q-1)}}r^s 
\end{eqnarray*}
Using $z=s/(s-1)$ we have that
\begin{eqnarray*}
\Vert{X_i}\Vert_{s}^s =   \frac{\Vert{C_i}\Vert_z^{s(q-1)}}{\Vert{C}\Vert_{z,q}^{s(q-1)}}r^s  .
\end{eqnarray*}

Thus,
\begin{eqnarray*}
\Vert{X_i}\Vert_{s} = \left({\frac{\Vert{C_i}\Vert_z}{\Vert{C}\Vert_{z,q}}}\right)^{q-1}r .
\end{eqnarray*}

Finally, we have that

\begin{eqnarray*}
&&\Vert{X}\Vert_{s,p}^p = \sum_{i\in[m]}\Vert{X_i}\Vert_s^p = 
\sum_{i\in[m]}\left({\frac{\Vert{C_i}\Vert_z}{\Vert{C}\Vert_{z,q}}}\right)^{p(q-1)}r^p = \\
&&\sum_{i\in[m]}\left({\frac{\Vert{C_i}\Vert_z}{\Vert{C}\Vert_{z,q}}}\right)^qr^p =
 \frac{r^p}{\Vert{C}\Vert_{z,q}^q}\sum_{i\in[m]}\Vert{C_i}\Vert_z^q = r^p .
\end{eqnarray*}
Thus, $X\in\ball_{s,p}(r)$.
\end{proof}

\end{document}